\documentclass{article}
\usepackage{braket, amssymb, amsmath, amsthm, graphicx, subcaption, tikz, longtable, multirow, cite}
\usepackage[colorlinks = true]{hyperref} 
\usepackage[margin = 2cm]{geometry}

\newtheorem{lemma}{Lemma}
\newtheorem{example}{Example}
\newtheorem{theorem}{Theorem}
\newtheorem{corollary}{Corollary}
\newtheorem{definition}{Definition}

\DeclareMathOperator{\bin}{bin}
\DeclareMathOperator{\m}{mod}

\title{A Boolean Function Theoretic Approach to Quantum Hypergraph States and Entanglement}
\author{Supriyo Dutta \\ Department of Mathematics\\ National Institute of Technology Agartala 
	\\Jirania, West Tripura, India - 799046 
	\\ Email: \texttt{dosupriyo@gmail.com}}
\date{} 

\begin{document}
	\maketitle 
	
	\begin{abstract}
		The hypergraph states are pure multipartite quantum states corresponding to a hypergraph. It is an equal superposition of the states belonging to the computational basis. Given any hypergraph, we can construct a hypergraph state determined by a Boolean function. In contrast, we can find a hypergraph, corresponding to a Boolean function. This investigation develops a number of combinatorial structures concerned with the hypergraph states. For instance, the elements of the computational basis generate a lattice. The chains and antichains in this lattice assist us to find the equation of the Boolean function explicitly as well as to find a hypergraph. In addition, we investigate the entanglement property of the hypergraph states in terms of their combinatorial structures. We demonstrate several classes of hypergraphs, such that every cut of equal length on the corresponding hypergraph states has an equal amount of entanglement.
		
		\noindent{\bf Keywords:} Quantum hypergraph state, Lattice, Boolean function, Entanglement, Partial transpose.
	\end{abstract}

	\section{Introduction}
	
		Quantum information is a rapidly expanding field of research due to its theoretical successes in super-dense quantum coding \cite{bennett1992communication, nielsen2002quantum}, teleportation \cite{bennett1993teleporting}, quantum error correction \cite{devitt2013quantum}, and many more \cite{mcmahon2007quantum, wilde2013quantum, benenti2007principles}. Most of these protocols utilize entanglement in quantum states as a resource \cite{horodecki2009quantum}. Multipartite pure quantum states expand the applicability of these schema. The quantum hypergraph states \cite{qu2013encoding, rossi2013quantum} which can be considered as a generalization of graph state \cite{hein2004multiparty, anders2006fast, van2005local} or cluster states \cite{briegel2001persistent, nielsen2006cluster} is an important class of multipartite, pure, entangled quantum states. They are equal superposition of all possible states in the computational basis of a given order.
		
		The quantum hypergraph states are spread into non-unitary equivalent classes \cite{qu2013multipartite}. They have crucial connections with stabilizer groups in the theory of quantum error correcting codes \cite{wagner2018analysis}. There is a relation between quantum hypergraph states and the local Pauli group \cite{rossi2013quantum}. Quantum entanglement is used as a resource in quantum information and computation. The quantum information theoretic community has a deep interest in the nature of entanglement in quantum hypergraph states \cite{dutta2019permutation, zhou2022entanglement, amouzou2022entanglement, sarkar2022geometry}. Entanglement in continuous variable hypergraph states is studied in \cite{moore2019quantum}. A projective multipartite entanglement witness is introduced in \cite{ghio2017multipartite}. It can be proved that all quantum hypergraph states are maximally entangled \cite{qu2013relationship}. An entropic measure of entanglement is studied for quantum hypergraph states in \cite{qu2014entropic}. A number of non-classical properties of quantum hypergraph states are illustrated in \cite{guhne2014entanglement}. They are also utilized in different applications in quantum information and computation \cite{zhu2019efficient, xiong2018qudit, toth2005detecting, takeuchi2019quantum, takeuchi2018verification, gachechiladze2019changing, tao2022verification, banerjee2020quantum}.
		
		Although a number of research works have been published in recent years, the combinatorial facets of these states are not ``well-understood" to date. These states have a number of beautiful combinatorial characteristics, which are essential in studying the entanglement of quantum hypergraph states. A hypergraph state is determined by a Boolean function. The Boolean functions are necessary for classical and quantum information \cite{galindo2002information}. To the best of our knowledge, this is the first study to make a powerful connection between the theory of hypergraphs, Boolean functions, and quantum states. We mention a constructive procedure to generate the Boolean function when a hypergraph state is given, as well as the hypergraph when a Boolean function is given. These combinatorial properties are useful in studying entanglement in hypergraph states. We demonstrate a number of hypergraph states with equal entanglement for all possible bipartitions.
		
		The article is divided into four sections. In Section 2, we present the necessary ideas to read this article, which are the poset and lattice, hypergraphs, and multipartite entanglement. Here we justify that there is a one-to-one correspondence between the Boolean functions and hypergraph states. Section 3 is dedicated to a study of combinatorial properties and entanglement in quantum hypergraph states. In this section, we construct a number of quantum hypergraph states with equal entanglement in all possible bipartitions. Then, we conclude this article.

	\section{Preliminary}
		
		In this section, we discuss a few mathematical concepts which are necessary to elaborate the rest of the article. A set $X$ is a well-defined collection of district objects. The Cartesian product of two sets $X$ and $Y$ is a new set $X \times Y = \{(x, y); x \in X, y \in Y\}$. A relation $R$ on a set $X$ is a subset of the Cartesian product $X \times X$. A relation $R$ is reflexive if $(x, x) \in R$ for all $x \in X$. It is transitive if for $(x, y) \in R$ and $(y, z) \in R$ we have $(x, z) \in R$. It is anti-symmetric if $(x, y) \in R$ and $(y, x) \in R$ indicate $x = y$. The relation $R$ is a partial order relation if it is reflexive, antisymmetric and transitive. A poset is a combination of a set $X$ and a partial order relation $\subset$ defined on it. A maximal element of a poset is an element $x$ such that if $x \subset y$, then $y = x$. In a poset, $z$ is a lower bound of $x$ and $y$ if $z \subset x$ and $z \subset y$. A greatest lower bound of $x$ and $y$ is a maximal element of the set of lower bounds. A lattice is a poset in which each pair of elements has a unique greatest lower bound and a unique least upper bound. A chain is a subset of a poset such that any two elements are related. \cite{crawley1973algebraic, lidl2012applied}
		
		A hypergraph $G = (V(G), E(G))$ consists of a set of vertices $V(G) = \{1, 2, \dots n\}$ and a set of hyperedges $E(G) = \{e : e \subset V(G)\}$ \cite{bretto2013hypergraph}. The  non-empty hyperedge $e = (v_1, v_2, \dots v_{|e|})$ contains vertices $v_1, v_2, \dots v_{|e|}$. A $k$-graph is a hypergraph, such that every hyperedge contains exactly $k$ vertices. A complete $k$-graph contains all possible hyperedges with $k$ vertices. The vertex-edge incidence matrix $M = (m_{v,e})_{n \times m}$ of $G$ with $n$ vertices and $m$ hyperedges is given by $m_{v, e} = 1$ if  $v \in e$ and $m_{v, e} = 0$ otherwise. Clearly, $\sum_{v = 1}^n m_{v, e} = |e|$.  The minimum number of vertices in a hyperedge is the co-rank of a hypergraph. 
		
		Now we are in a position to introduce quantum hypergraph states. Any integer $i \in [2^n] = \{0, 1, 2,\dots, $ $(2^n - 1)\}$ has an $n$-term binary representation $\bin(i) = (i_1, i_2, \dots i_n)$ where $i_j \in \{0, 1\}$ and $j = 1, 2, \dots n$.  A normalised vector in two-dimensional Hilbert space $\mathcal{H}_2$ is a qubit. Denote the vectors $\ket{0} = \begin{bmatrix}1 \\ 0 \end{bmatrix}, \ket{1} = \begin{bmatrix}0 \\ 1 \end{bmatrix}$ and $\ket{+} = \frac{1}{\sqrt{2}}(\ket{0} + \ket{1}) = \frac{1}{\sqrt{2}}\begin{bmatrix}1 \\ 1 \end{bmatrix}$. The Pauli $Z$ operator is given by $Z = \begin{bmatrix} 1 & 0 \\ 0 & -1 \end{bmatrix}$, such that $Z\ket{0} = \ket{0}$ and $Z\ket{1} = -\ket{1}$. An $n$-qubit state is a normalised vector in $\mathcal{H}_2^{\otimes n} = \mathcal{H}_2 \otimes \dots \otimes \mathcal{H}_2 (n$-times). The set $\mathcal{B} = \{\ket{\bin(i)} = \ket{i_1}\otimes \ket{i_2} \otimes \dots \otimes \ket{i_n}: i \in [2^n]\}$ forms the computational basis of $\mathcal{H}^{\otimes n}$. Note that, $\ket{+}^{\otimes n} = \ket{+} \otimes \ket{+} \otimes \dots \otimes \ket{+} = \frac{1}{\sqrt{2^n}} \sum_{i \in [2^n]}\ket{\bin(i)}$.
		
		There is a one-to-one correspondence between the elements in $\bin(i)$ for any $i \in [2^n]$ and $V(G)$. Hence, we can write $\ket{\bin(i)} = \ket{i_{v_1}} \otimes \ket{i_{v_2}} \otimes \dots \otimes \ket{i_{v_n}}$ , where $i_{v_i} \in \{0, 1\}$. Corresponding to a hyperedge $e = (v_1, v_2, \dots v_{|e|})$ define a controlled-$Z$ operator $Z_e$ acting on $\mathcal{H}_2^{\otimes n}$, such that
		\begin{equation}\label{hyperedge_operator}
			Z_e \ket{\bin(i)} = \begin{cases} - \ket{\bin(i)} & ~\text{if}~ i_{v_j} = 1 ~\text{for all}~ v_j \in e; \\ \ket{\bin(i)} & ~\text{if}~ i_{v_j} = 0 ~\text{for any}~ v_j \in e. \end{cases}
		\end{equation}
		Clearly, $\ket{\bin(i)}$ are the eigenvectors of $Z_e$ with eigenvalues $\pm 1$ depending on $e$. Therefore, with respect to the basis $\mathcal{B}$, $Z_e$ is a diagonal matrix of order $2^n$ whose diagonal entries are $1$, or $-1$. Also $Z_e Z_e^\dagger = I$. Hence, the set of all controlled-Z operators $\{Z_e: e \in E(G)\}$ forms a family of commuting normal matrices. Now corresponding to a hypergraph $G$ there is a unique operator $\mathcal{U} = \prod_{e \in E(G)}Z_e$, which leads us to the definition of quantum hypergraph states.
		\begin{definition}
			Given any hypergraph $G$ with $n$ vertices, there is a unique $n$-qubit quantum state, called hypergraph state, which is denoted by $\ket{G}$ and defined by $\ket{G} = \mathcal{U} \ket{+}^{(\otimes n)}$.
		\end{definition}
		Equation (\ref{hyperedge_operator}) suggests that $\mathcal{U}$ only alters the sign of $\ket{\bin(i)}$ depending on $E(G)$. Now, expanding $\ket{G}$ we get 
		\begin{equation}\label{hypergraph_state}
			\ket{G} = \mathcal{U} \ket{+}^{\otimes n} = \frac{1}{\sqrt{2^n}} \sum_{i \in [2^n]} \mathcal{U} \ket{\bin(i)} = \frac{1}{\sqrt{2^n}}\sum_{i \in [2^n]} (-1)^{f(\bin(i))} \ket{\bin(i)},
		\end{equation}
		where $f:\{0, 1\}^{\times n} \rightarrow \{0, 1\}$ is a Boolean function depending on $G$. Recall that, a Boolean function with $n$ variables is a function $f: \{0, 1\}^n \rightarrow \{0, 1\}$ \cite{crama2011boolean}. Note that the expression of $\ket{G}$ contains all the elements of $\mathcal{B}$, with coefficients $\pm 1$. As $\mathcal{B}$ contains $2^n$ elements there are $2^{2^n}$ different hypergraph states $\ket{G}$. Also, the number of all possible Boolean functions is $2^{2^n}$. Hence, there is a one-to-one correspondence between the set of hypergraph states and the set of all Boolean functions with $n$ variables. We can write it as the following lemma:
		
		\begin{lemma}
			Any Boolean function over $n$ variables corresponds to an $n$-qubit hypergraph state.
		\end{lemma}
		
		A hypergraph state can also be represented by its density matrix
		\begin{equation}\label{density_matrix}
			\rho_G = \ket{G}\bra{G} = \frac{1}{2^n} \sum_{i \in [2^n]} \sum_{j \in [2^n]} (-1)^{f(\bin(i)) + f(\bin(j))} \ket{i_1 i_2 \dots i_n} \bra{j_1 j_2 \dots j_n}.
		\end{equation}
		Therefore, if $\rho_G = (\rho_{i,j})_{2^n \times 2^n}$, then $\rho_{i,j} = (-1)^{f(\bin(i)) + f(\bin(j))}$. 
		
		In this work, we also investigate the separability properties of hypergraph states. A cut set $T = \{k_1, \dots k_m: m < n\}$ is a set of indices, where, $m$ is called the length of $T$. All the qubits in $\ket{\bin(i)}$ with indices in $T$ belongs to the Hilbert space $\mathcal{H}_2^{(T)}$. Similarly, the qubits in $\ket{\bin(i)}$ whose indices are not is $T$ belongs to $\mathcal{H}_2^{(\overline{T})}$. Clearly, $\mathcal{H}_2^n = \mathcal{H}_2^{(T)} \otimes \mathcal{H}_2^{(\overline{T})}$. Depending on $T$ we can define a permutation $P$ such that $P(k_i) = i$. Therefore, given any set of indices $T$ the permutation $P$ maps the qubits corresponding to $T$ of $\ket{\bin(i)}$ to the set of qubits corresponding to $\mathcal{I} = \{1, \dots m: m \le n\}$. We can express $\ket{\bin(i)} = \bigotimes_{v \in V_T} \ket{i_v} \otimes \bigotimes_{v \in V_{\overline{T}}} \ket{i_v} = P \left[ \bigotimes_{j = 0}^{m - 1} \ket{i_j} \otimes \bigotimes_{k = m}^n \ket{i_k} \right]$. 
		
		A quantum state $\ket{\psi}$ is separable with respect to the cut set $T$ if $\ket{\psi} = \ket{\psi_1} \otimes \ket{\psi_2}$ where $\ket{\psi_1} \in \mathcal{H}_2^{(T)}$ and $\ket{\psi_2} \in \mathcal{H}_2^{(\overline{T})}$. Otherwise $\ket{\psi}$ is entangled. A fully entangled state is entangled for all possible cut sets. The partial transpose with respect to the cut set $T$ on $\rho_G$ is given by,
		\begin{equation}\label{partial_transpose}
			\begin{split} 
				\rho^{\tau_T}_G = \frac{1}{2^n} \sum_{i \in [2^n]} \sum_{i \in [2^n]} (-1)^{f(\bin(i)) + f(\bin(j))} & \ket{i_1 \dots i_{k_1 - 1} j_{k_1} i_{k_1 + 1} \dots i_{k_{m} - 1} j_{k_m} i_{k_m + 1} \dots i_n}\\
				& \bra{j_1 \dots j_{k_1 - 1} i_{k_1} j_{k_1 + 1} \dots j_{k_m - 1} i_{k_m} j_{k_m + 1} \dots j_n}.
			\end{split} 
		\end{equation}

	\section{Hypergraph quantum states, Boolean functions, and Entanglement} 
		
		In this section, first, we represent the function $f(\bin(i))$ in terms of the combinatorial structures in the hypergraph $G$. Then we elaborate on how entanglement depends on hypergraphs. In the end, we discuss a few hypergraphs having equal amount of entanglement with respect to every cut. We begin with the following lemma illustrating the relation between hyperedges and $\ket{\bin(i)}$.
		
		\begin{lemma}
			For any hyperedge $e$ there are $2^{n - |e|}$ integers $i \in [2^n]$, such that, $Z_e \ket{\bin(i)} = -\ket{\bin(i)}$.
		\end{lemma}
		
		\begin{proof}
			Consider a hyperedge $e = (v_1, v_2, \dots v_{|e|})$. Equation (\ref{hyperedge_operator}) indicates that $Z_e \ket{\bin(i)} = - \ket{\bin(i)}$ holds if $i_{v_1} = i_{v_2}= \dots = i_{v_{|e|}} = 1$. Remaining $n - |e|$ elements of $\bin(i)$ may be selected in $2^{n - |e|}$ different ways which are either $0$ or $1$. Hence, given any hyperedge $e$, there are $2^{n - |e|}$ integers $i \in [2^n]$, such that $Z_e \ket{\bin(i)} = -\ket{\bin(i)}$.
		\end{proof}
	
		For an integer $i \in [2^n]$, we define \textit{a set of vertices} $\mathcal{O}(i) = \{v: i_v = 1, i_v \in \bin(i)\}$. Clearly, $\mathcal{O}(0) = \emptyset$ and $\mathcal{O}(2^n - 1) = V(G)$. Also, for any hyperedge $e$ there exists a unique integer $i$, such that $e = \mathcal{O}(i)$. Define a binary relation $\subset$ on the set $[2^n]$, such that $i \subset j$ if $\mathcal{O}(i) \subset \mathcal{O}(j)$, which is a partial order relation. Note that, $([2^n], \subset)$ is an algebraic lattice. We also define \textit{a set of hyperedges} $E_i = \{e: e \in E(G), e \subset \mathcal{O}(i)\}$. Note that, for a hyperedge $e = (v_1, v_2, \dots v_{|e|})$ there is an integer $k$, such that $\mathcal{O}(k) = e$. The non-empty set of integers $(e) = \{i: e \subset \mathcal{O}(i)\}$ corresponds a chain under the partial order relation $\subset$, with infimum $|e|$ and supremum $(2^n - 1)$. In the following theorem, we write the function $f(\bin(i))$ in terms of $E_i$.
		
		\begin{theorem} \label{counting_1}
			Let $f: \{0, 1\}^{\times n} \rightarrow \{0, 1\}$ be the Boolean function determining the quantum hypergraph state $\ket{G}$. Then, $f(\bin(i)) = |E_i|(\m 2)$.
		\end{theorem}
		
		\begin{proof} 
			Recall that, given any integer $i \in [2^n]$ and a hyperedge $e$, $Z_e\ket{\bin(i)} = -\ket{\bin(i)}$ holds, if $e \subset \mathcal{O}(i)$, defined above. The operator $\mathcal{U}$ is the product of all operators $Z_e$. Now only the hyperedges belonging to the set $E_i$ alter the sign of $\ket{\bin(i)}$ when we apply $\mathcal{U}$ on $\ket{\bin(i)}$. Therefore, $f(\bin(i)) = |E_i| (\m 2)$. 
		\end{proof}
	
		An alternative expression $f(\bin(i))$ can also be determined by the incidence matrix of the hypergraph as follows:
		
		\begin{theorem}
			Let $M$ be the vertex-edge incidence matrix of a hypergraph $G$ and $k$ be the number of zeros in the vector $M^t(j_n - \bin(i)^t)$, where $j_n$ be the all one column vector of dimension $n$. Then $f(\bin(i)) = k(\m 2)$.
		\end{theorem}	
		
		\begin{proof} 	
			Definition of incidence matrix indicates that $M^tj_n = (|e_1|, |e_2|, \dots |e_m|)^t$. A hyperedge $e = (v_1, v_2, \dots v_{|e|})$ corresponds to the $e$-th row in $M^t$ which has $1$ in the $v_1, v_2, \dots v_{|e|}$-th positions and remaining all elements are zeros. The $e$-th row of $M^tj_n$ is $|e|$. If $e \subset \mathcal{O}(i)$, then $Z_e\ket{\bin(i)} = -\ket{\bin(i)}$. For this $i$, the $e$-th element of $M^t\bin(i)$ is again $|e|$. Hence, the $e$-th element of $M^t(j_n - \bin(i)) = 0$. Now $\mathcal{U}$ is the product of all these controlled-$Z$ operators. The number of hyperedges $e$ such that $Z_e$ alters the sign of $\ket{\bin(i)}$ is $k$, which is the number of zero elements in $M^t(j_n - \bin(i)^t)$. Therefore, $f(\bin(i)) = k(\m 2)$.
		\end{proof}
		
		Now, we study entanglement in $\ket{G}$ in terms of the vertices and hyperedges in $G$. A cut set $T$ also partitions $V(G)$ into two subsets: $V_T = \{v_k: k \in T\}$ and $V_{\overline{T}} = \{v_k: k \notin T\}$. We have the following observation:
		
		\begin{theorem}
			The quantum hypergraph state $\ket{G}$ is entangled with respect to the cut set $T = \{k_0, k_1, \dots k_{(m - 1)}\}$ if there is a hyperedge $\epsilon \in E(G)$, such that, $\epsilon \cap V_T \neq \emptyset$ and $\epsilon \cap V_{\overline{T}} \neq \emptyset$.
		\end{theorem}
		
		\begin{proof}
			Let $\ket{G}$ be a seperable state with respect to the cut $T$, that is there are $\ket{\psi_1} \in \mathcal{H}^{(T)}$ and $\ket{\psi_2} \in \mathcal{H}^{(\overline{T})}$, such that $\ket{G} = \ket{\psi_1} \otimes \ket{\psi_2}$. In terms of the computational basis
			\begin{equation}
				\ket{\psi_1} = \sum_{j = 0}^{2^m-1} a_j \ket{\bin(j)} ~\text{and}~ \ket{\psi_2} = \sum_{k = 0}^{2^{(n - m)} - 1} b_k \ket{\bin(k)},
			\end{equation}
			where $\sum_{j = 0}^{2^m-1} |a_j|^2 = 1$, and $\sum_{k = 0}^{2^{(n - m)} - 1} |b_k|^2 = 1$. Here, $\bin(j)$ is $m$-term binary representation of $j$, as well as $\bin(k)$ is $(n - m)$ term binary representation of $k$. Consider the integer $i \in [2^n]$, such that $\mathcal{O}(i) = \epsilon$. Now there are $\ket{\bin(j)}$ and $\ket{\bin(k)}$ in $\mathcal{H}^{(T)}$ and $\mathcal{H}^{(\overline{T})}$ satisfying $\ket{\bin(i)} = \ket{\bin(j)}\otimes \ket{\bin(k)}$. Comparing the coefficients of $\ket{G}$ and $\ket{\psi_1} \otimes \ket{\psi_2}$, after expanding them in computation basis we have $a_j b_k = \frac{1}{\sqrt{2^n}}(-1)^{f(\bin(i))}$.
			
			We have seen that every hyperedge $e$ introduces a chain in $([2^n], \subset)$ which is denoted by $(e)$. Theorem \ref{counting_1} suggests that, $f(\bin(i)) = |E_i|(\m 2)$ where $E_i = \{e: e \in E(G), e \subset \mathcal{O}(i)\}$. The cut set $T$ partitions $E_i$ into three classes: $E_1 = \{e: e \subset V_T\}$, $E_2 = \{e: e \subset V_{\overline{T}}\}$, and $E_3 = \{e: e \cap V_T \neq \emptyset, e \cap V_{\overline{T}} \neq \emptyset\}$. Clearly, $|E_i| = |E_1| + |E_2| + |E_3|$. Therefore, $a_j b_k = \frac{1}{\sqrt{2^n}}(-1)^{(|E_1| + |E_2| + |E_3|)}$. 
			
			Considering the coefficients of $\ket{\bin(j)} \otimes \ket{0}^{\otimes (n - m)}$ and $\ket{0}^{\otimes m} \otimes \ket{\bin(k)}$ in both side of the equation $\ket{G} = \ket{\psi_1} \otimes \ket{\psi_2}$, we get $a_jb_0 = \frac{1}{\sqrt{2^n}}f(\bin(j), 0, 0, \dots 0 (n-m)$-times), and $a_0b_k = \frac{1}{\sqrt{2^n}}f(0, 0, \dots 0 (m\text{- times}), \bin(k))$. Observe that, $Z_e$ $\ket{\bin(j), 0, 0, \dots 0} = - \ket{\bin(j), 0, 0, \dots 0}$, if and only if $e \in E_1$. Note that, $f(\bin(j), 0, 0, \dots 0 (n-m)$-times) $ = |E_1|$, and similarly $f(0, 0, \dots 0 (m\text{- times}),$ $ \bin(k)) = |E_2|$. Therefore $a_jb_0 = \frac{1}{\sqrt{2^n}}(-1)^{|E_1|}$, and $a_0b_k = \frac{1}{\sqrt{2^n}}(-1)^{|E_2|}$. The leading term of $\ket{G} = \ket{\psi_1} \otimes \ket{\psi_2}$ indicates that $a_0b_0 = \frac{1}{\sqrt{2^n}}$. Therefore, $a_j b_k = \frac{1}{\sqrt{2^n}}(-1)^{(|E_1| + |E_2|)}$, which contradicts to the existence of $\epsilon$. Therefore $\ket{G}$ is entangled with respect to the cut set $T$.
		\end{proof}
		
		Now, we concentrate on the algebraic lattice $([2^n], \subset)$. We partition $[2^n]$ into $n$ subsets as follows:
		\begin{equation}
			C_k = \{i: |\mathcal{O}(i)| = k\}, ~\text{where}~ k = 0, 1, \dots n.
		\end{equation}
		Note that every $C_k$ forms an antichain in $([2^n], \subset)$ because no two elements in $C_k$ are related by the partial order relation $\subset$. For instance, the elements in $[2^4]$ can be partitioned as $C_0 = \{0\}, C_1 = \{1, 2, 4, 8\}$, $C_2 = \{3, 5, 6, 9, 10, 12\}, C_3 = \{7, 11, 13, 14\}$, and $C_4 = \{15\}$.
	
		\begin{lemma}\label{true_lemma}
			There are ${s \choose k}$ different choices for $i \in C_k$ for any given $j \in C_s$ such that $i \subset j$ where $s > k$.
		\end{lemma}
		
		\begin{proof}
			As $j \in C_s$, $\bin(j)$ contains $s$ number of $1$. Now $i \subset j$ if the $1$s in $\bin(i)$ share same position in $\bin(j)$. The number of $1$ in $\bin(j)$ is $s$. We can select $k$ positions in $\bin(i)$ from $s$ positions in ${s \choose k}$ ways. Therefore, there are ${s \choose k}$ choices of $i \in C_k$ for any given $j \in C_s$
		\end{proof}
	
		Now we consider a few particular hypergraphs and the Boolean functions in their corresponding states. 
		\begin{corollary}\label{only_one}
			Let $G = (V(G), E(G))$ be a hypergraph with $n$ vertices and $E(G) = \{e\}$ where $e = V(G)$, then 
			\begin{equation}
				\ket{G} = \frac{1}{\sqrt{2^n}} \left[\sum_{i = 0}^{2^n - 2} \ket{\bin(i)} - \ket{11 \dots 1} \right].
			\end{equation}
		\end{corollary}
		
		\begin{proof} 
			The hypergraph has only one hyperedge $e = V(G)$. Therefore, $E_i = \emptyset$ for $i = 0, 1, \dots 2^n - 2$ and $E_{2^n - 1} = V(G)$. Hence, $f(\bin(i)) = 1$ for $i = 2^n - 1$; and $f(\bin(i)) = 0$ otherwise. It leads us to the result.
		\end{proof}
			
		\begin{corollary} \label{complete_k_graph}
			Given a complete $k$-graph $G$ with $n$ vertices the Boolean function $f$ is represented by
			\begin{equation}
				f(\bin(i)) = \begin{cases} 0 & ~\text{if}~ i \in C_s, s < k; \\ {{s \choose k} \m 2} & ~\text{if}~ i \in C_s, s \geq k. \end{cases}
			\end{equation}
		\end{corollary}
		
		\begin{proof}
			There is no hyperedge with cardinality less than $k$. Hence, there is no operator $Z_e$, such that $Z_e\ket{\bin(i)} = -\ket{\bin(i)}$, for $|\mathcal{O}(i)| < k$; that is $f(\bin(i)) = 1$ for all $i \in C_s, s < k$. All possible hyperedges of cardinality $k$ are available in $E(G)$. Hence, for any $i \in C_k$, there is an operator $Z_e$, such that $Z_e\ket{\bin(i)} = - \ket{\bin(i)}$. Therefore, $f(\bin(i)) = 1 = {k \choose k}(\m 2)$, for all $i \in C_k$. Also, there is no hyperedge of cardinality greater than $k$. Applying Lemma \ref{true_lemma} we state that for any element $j$ in $C_s, s > k$ there are ${s \choose k}$ elements $i$ in $C_k$, such that $i \subset j$. All these relations correspond to a controlled-Z operator acting on $\ket{\bin(i)}$. Hence, the operator $\mathcal{U}$ generates $-1$ for ${s \choose k}$ times in the coefficient of $\ket{\bin(i)}$. Therefore,  $f(\bin(i)) = {{s \choose k} \m 2}$ for $i \in C_s, s > k$.
		\end{proof}
		
		If $G$ is a complete graph, more precisely a complete $2$-graph with $n$ vertices, then 
		\begin{equation}
			f(\bin(i)) = \begin{cases} 0 & ~\text{if}~ i \in C_s, s < 2; \\ {{s \choose 2} \m 2} & ~\text{if}~ i \in C_s, s \geq 2. \end{cases}
		\end{equation}

		\begin{corollary}\label{complete_k_mixed}
			Let $G$ be a hypergraph such that, the set of hyperedges $E(G) = E_1 \cup E_2 \cup \dots \cup E_p$ where the hypergraph $(V(G), E_s)$ is a complete $k_s$-graph for distinct values of $s = 1, 2, \dots p$ and $0 < k_1 < k_2 \dots < k_p$. Then the Boolean function corresponding to the hypergraph $G$ can be written as 
			$$f(\bin(i)) = \begin{cases} 0 & ~\text{if}~ i \in C_s, s < k_1; \\ {s \choose k_1} \m 2 & ~\text{if}~ i \in C_s, k_1 \leq s < k_2; \\ {s \choose k_1} + {s \choose k_2} \m 2 & ~\text{if}~ i \in C_s, k_2 \leq s < k_3; \\ \vdots \\ {s \choose k_1} + {s \choose k_2} + \dots + {s \choose k_p} \m 2 & ~\text{if}~ i \in C_{k_p}. \end{cases}$$
		\end{corollary}
		
		\begin{proof}
			Using Corollary \ref{complete_k_graph} we can justify that
			$$f(\bin(i)) = \begin{cases} 0 & ~\text{if}~ i \in C_s, s < k_1; \\ {s \choose k_1} \m 2 & ~\text{if}~ i \in C_s, k_1 \leq s < k_2. \end{cases}$$
			An element $i \in C_s$ for $s \geq k_2$ has connections with the elements in $C_{k_1}$ and $C_{k_2}$. Each of these connections introduces a controlled-Z operator on the elements of $C_s$. Therefore, we may conclude that, $f(\bin(i)) = {s \choose k_1} + {s \choose k_2} \m 2 ~\text{if}~ i \in C_s, k_2 \leq s$. Similarly, we can extend it for larger values of $p$.
		\end{proof}
		
		If co-rank of the hypergraph $G$ is $r$, then there is an operator $Z_e$, such that $Z_e\ket{\bin(i)} = -\ket{\bin(i)}$ where $i \in C_r$. The hypergraph does not contain any hyperedge of cardinality less than $r$. Thus, $f(\bin(i)) = 0$ for $i \in C_s$, where $s < r$.
		
		We mentioned that there is a one-to-one correspondence between the set of hypergraph states and the set of the Boolean functions. Therefore, if a Boolean function $f: \{0, 1\}^{\times n} \rightarrow \{0, 1\}$ is given one can construct the corresponding hypergraph $G$. Clearly, $G$ has $n$ vertices. Determining a hyperedge in $E(G)$ is equivalent to identifying one chain $(e)$ in the algebraic lattice $([2^n], \subset)$. The Boolean function $f$ can also be expressed in terms of these chains. If an integer $i$ belongs to an intersection of $k$ different chains, $f(\bin(i)) = k(\m 2)$. The following example will illustrate our approach to determine the hyperedges in $E(G)$.
		
		\begin{example}
			The Horn function is a well-known class of Boolean functions. Consider a Horn function $f(x_1, x_2, x_3) = x_1 \overline{x_2} ~\vee~ x_1 \overline{x_3} ~\vee~ x_2x_3$. The truth table of this function is given below
			\begin{center} 
				\begin{tabular}{|c c c| p{.6cm} |}
					\hline 
					$x_1$ & $x_2$ & $x_3$ & ~ f \\
					\hline 
					$0$ & $0$ & $0$ & ~ $0$ \\
					$0$ & $0$ & $1$ & ~ $0$ \\
					$0$ & $1$ & $0$ & ~ $0$ \\
					$0$ & $1$ & $1$ & ~ $1$ \\
					$1$ & $0$ & $0$ & ~ $1$ \\
					$1$ & $0$ & $1$ & ~ $1$ \\
					$1$ & $1$ & $0$ & ~ $1$ \\
					$1$ & $1$ & $1$ & ~ $1$ \\
					\hline 
				\end{tabular}
			\end{center} 
			Therefore, the corresponding quantum hypergraph state is:
			\begin{equation}
				\ket{G} = \frac{1}{\sqrt{2^n}}\left[\ket{000} + \ket{001} + \ket{010} - \ket{011} - \ket{100} - \ket{101} - \ket{110} - \ket{111}\right].
			\end{equation}
			As $\ket{G}$ is a three qubit state, the hypergraph $G$ has three vertices that is $V(G) = \{v_1, v_2, v_3\}$. We can partition $[2^3]$ into subsets $C_0, C_1, C_2$ and $C_3$, whose entries can be represented in binary form as $\{(0, 0, 0)\}, \{(0, 0, 1), (0, 1, 0), (1, 0, 0)\}$, $\{(0, 1, 1), (1, 0, 1), (1, 1, 0)\}$, and $\{1, 1, 1\}$, respectively. In the expression of $\ket{G}$, the coefficient of $\ket{100}$ has a negative sign. Also, $(1, 0, 0) \in C_1$ and no other element of $C_0$ or $C_1$ has negative coefficient. Note that, $Z_{e_1}\ket{100} = -\ket{100}$ holds if $e_1 = (v_1) \in E(G)$. Elements in the chain $(e_1) = \{i: e_1 \subset \mathcal{O}(i)\}$ in binary form are $\{(1, 0, 1), (1, 1, 0), (1, 1, 1)\}$. Coefficients of $\ket{101}, \ket{110}$, and $\ket{111}$ are negative in $\ket{G}$. Also, the coefficient of $\ket{011}$ is negative but $(0, 1, 1) \notin (e_1)$. Note that, $Z_{e_2}\ket{011} = -\ket{011}$ holds if $e_2 = (v_2, v_3) \in E(G)$. But $(1, 1, 1) \in (e_1)\cap (e_2)$, and $Z_{e_1} Z_{e_2} \ket{111} = \ket{111}$. Therefore, we need one more hyperedge $e_3 = (v_1, v_2, v_3)$ such that $Z_{e_3}\ket{111} = -\ket{111}$. Now, note that $\mathcal{U}\ket{+}^{\otimes 3} = \ket{G}$ where $\mathcal{U} = Z_{e_1} Z_{e_2} Z_{e_3}$. Therefore, $E(G) = \{(v_1), (v_2, v_3), (v_1, v_2, v_3)\}$. The lattice $([2^3], \subset)$ and the chains $(e_1), (e_2), (e_3)$ are depicted in the figure below:
			\begin{center}
				\begin{tikzpicture}[scale = 2]
					\node at (0, 0) {$(0, 0, 0)$};
					\node at (0, 1) {$(0, 1, 0)$};
					\node at (-2, 1) {$(1, 0, 0)$};
					\node at (2, 1) {$(0, 0, 1)$};
					\node at (-2, 2) {$(1, 1, 0)$};
					\node at (0, 2) {$(1, 0, 1)$};
					\node at (2, 2) {$(0, 1, 1)$};
					\node at (0, 3) {$(1, 1, 1)$};
					\node at (-4, 0) {$C_0$};
					\node at (-4, 1) {$C_1$};
					\node at (-4, 2) {$C_2$};
					\node at (-4, 3) {$C_3$};
					\draw (-2, 1.2) -- (-2, 1.8);
					\node at (-2.3, 1.5) {$(e_1)$};
					\draw (-2, 1.2) -- (0, 1.8);
					\node at (-1.3, 1.6) {$(e_1)$};
					\draw (-2, 2.2) -- (0, 2.8);
					\node at (-1, 2.7) {$(e_1)$};
					\draw (0, 2.2) -- (0, 2.8);
					\node at (-.3, 2.4) {$(e_1)$};
					\draw (2, 2.2) -- (0, 2.8);
					\node at (1.5, 2.6) {$(e_2)$};
					\node at (1, 3) {$(e_3)$};
				\end{tikzpicture}
			\end{center}
		\end{example}
		
		Till now we have discussed three different types of hypergraphs and their corresponding quantum states. Now we concentrate on their entanglement properties. 		
		\begin{lemma}\label{only_one_P}
			Let $P$ be the permutation which maps the qubits in the cut set $T$ to $\mathcal{I}$. Given a hypergraph $G$ with $n$ vertices and exactly one hyperedge that contains all the vertices, we have $P\ket{G} = \ket{G}$.
		\end{lemma}
		
		\begin{proof}
			Recall from corollary \ref{only_one} that $\ket{G} = \frac{1}{\sqrt{2^n}} \left[\sum_{i = 0}^{2^n - 2} \ket{\bin(i)} - \ket{11 \dots 1} \right]$. Note that, $P\ket{11 \dots 1} = \ket{11 \dots 1}$ for all permutation matrix $P$. Also, $P$ alters the elements of $\ket{\bin(i)}$ for $0 \leq i  \leq (2^n - 2)$ from one to another keeping their coefficients fixed. Therefore, $P\ket{G} = \ket{G}$.
		\end{proof}
		
		\begin{lemma}\label{completek_P}
			Let $P$ be the permutation which maps the qubits in the cut set $T$ to $\mathcal{I}$. Given a complete $k$-graph $G$ with $n$ vertices, we have $P\ket{G} = \ket{G}$.
		\end{lemma}
		
		\begin{proof}
			The hypergraph state corresponding to complete $k$-graph is discussed in Corollary \ref{complete_k_graph}. Expanding it we have
			\begin{equation}
			\ket{G} = \frac{1}{\sqrt{2^n}}\left[\sum_{i \in C_0\cup \dots C_{k-1}}\bin(i) - \sum_{i \in C_k}\bin(i) + (-1)^{k + 1 \choose k} \sum_{i \in C_{k + 1}} \bin(i) + \dots + (-1)^{n \choose k} \sum_{i \in C_n} \bin(i)\right]
			\end{equation}
			Recall that, for all clusters $C_s$, we have $|\mathcal{O}(i)| = |\mathcal{O}(j)|$ for any two $i$ and $j \in C_s$. Therefore, $\bin(i)$ and $\bin(j)$ have equal number of $1$s. The permutation $P$ does not alter the number of $1$s. It only replaces their respective positions. Hence, for any $i \in C_s$ there is $j$ in $C_s$ such that $P\ket{\bin(i)} = \ket{\bin(j)}$ for all $s$. Also, $f(\bin(i))$ has equal value for all $i \in C_s$ for a fixed $s$. Thus, we get $P\ket{G} = \ket{G}$. 
		\end{proof}
		
		In a similar fashion we can prove the following corollary:
		\begin{corollary}\label{mixed_k_P}
			Let $G$ be a hypergraph, such that the set of hyperedges $E(G) = E_1 \cup E_2 \cup \dots \cup E_p$ where the hypergraph $(V(G), E_s)$ is a complete $k_s$-graph for distinct values of $s = 1, 2, \dots p$ and $0 < k_1 < k_2 \dots < k_p$. Let $P$ be the permutation which maps the qubits in the set $T$ to $\mathcal{I}$. Then $P\ket{G} = \ket{G}$.
		\end{corollary}
		
		In the above results, we mention three classes of hypergraphs $G$ for which $P\ket{G} = \ket{G}$. The following theorem states that they are the only hypergraphs satisfying $P\ket{G} = \ket{G}$.
				
		\begin{theorem}
			Let $G$ be a hypergraph, such that $P\ket{G} = \ket{G}$, where $P$ is a permutation that maps the qubits in a set $T$ to $\mathcal{I}$. Then, we can express $E(G)$ as a union of the hyperedges of a number of complete $k$-graphs for distinct values of $k$.
		\end{theorem}
		
		\begin{proof}
			As $P\ket{G} = \ket{G}$, the operation $P$ keeps coefficients $\ket{\bin(i)}$ unaltered for all $i$. It also keeps a number of $1$s unchanged to map the qubits in a set $T$ to $\mathcal{I}$. Therefore, for all $i \in C_s$ there is $j \in C_s$ such that $P\ket{\bin(i)} = \ket{\bin(j)}$ for all $s$. To maintain an equal sign for all $\ket{\bin(i)}$ with $i \in C_s$, we need all the sets of hyperedges with an equal number of vertices. Therefore, we can express $E(G)$ as a union of the hyperedges of a number of complete $k$-graphs for distinct values of $k$.
		\end{proof}
		
		In the above theorem, when $k = n$, the hypergraph consists of a single hyperedge containing all the edges. It gives us the hypergraph mentioned in Corollary \ref{only_one} and Lemma \ref{only_one_P}. Complete $k$ graph is discussed in Corollary \ref{complete_k_graph} and Lemma \ref{completek_P}. Also, corollary \ref{complete_k_mixed} and \ref{mixed_k_P} considers the union of complete $k$-graphs. An interesting characteristic of these states is that two cuts of equal lengths have an equal amount of entanglement. We conclude it from the theorem below.
		
		\begin{theorem}\label{equivalance_of_partial_transpose}
			Let $G$ be a hypergraph with $n$ vertices, such that, $P\ket{G} = \ket{G}$, where $P: T \rightarrow \mathcal{I}$ is a permutation. Given any cut $T$ of length $m$, all the partial transposes of $\ket{G}$, $\rho^{\tau_T}_G$, have equal sets of eigenvalues.
		\end{theorem}
		
		\begin{proof}
			Let $\rho^{\tau_\mathcal{I}}_G$ and $\rho^{\tau_T}_G$ be the partial transposes of density matrix $\rho_G = \ket{G}\bra{G}$ with respect to the cut set $\mathcal{I} = \{0, 1, \dots (m-1)\}$, and $T =\{k_0, k_1, \dots k_{(m-1)}\}$, respectively. We need to prove $\rho^{\tau_\mathcal{I}}_G$ and $\rho^{\tau_T}_G$ share equal set of eigenvalues for any cut $T$ of length $m$. There is a permutation matrix $P$ such that $P\rho_GP^t$ is the density matrix of a state where $k_1, \dots k_m$ qubits $\rho_G$ act as $1, \dots m$ qubits. Let $\rho_1$ be the partial transpose on $P\rho_GP^t$ with respect to the cut set $\mathcal{I}$. Therefore, $P^t\rho_1P$ is the partial transpose on $\rho_G$ with respect to the cut set $T$. Symbolically, $\rho^{\tau_T}_G = P^t\rho_1P$. We consider that $P\ket{G} = \ket{G}$, that is $P\rho_G P^t = \rho_G$. Therefore, $\rho_1$ is partial transpose on $\rho_G$ with respect to $\mathcal{I}$. Symbolically, $\rho_1 = \rho^{\tau_{\mathcal{I}}}_G$. Combining all these we get, $\rho^{\tau_T}_G = P^t\rho^{\tau_{\mathcal{I}}}P$, for some unitary matrix $P$. Therefore, $\rho^{\tau_I}_G$ and $\rho^{\tau_J}_G$ have same set of eigenvalues.
		\end{proof}
		
		In \cite{brown2005searching}, a measures of entanglement is proposed based on the partial transpose. It considers the negated sum of all negative eigenvalues of $\rho^{\tau_T}$. The above theorem suggests that for all cuts of equal lengths entanglement is equal for those hypergraphs $G$ with $P\ket{G} = \ket{G}$.

	\section{Conclusion}
			
		Boolean algebra is extensively applied to quantum computation and information. This article initiates a new direction of research which involves combinatorial tools to analyze the properties of quantum states in terms of Boolean functions. We elaborate on a connection between the Boolean functions and the hypergraph states. Using this connection we have discussed the relations between the coefficients of the hypergraph state and the structure of underlined hypergraph for a number of hypergraph classes, such as the complete $k$ graph and their generalizations. We have identified different classes of hypergraphs whose corresponding quantum states have equal entanglement for cuts with equal lengths. The connection between the Boolean functions and hypergraph states is interesting due to the application of Boolean functions in classical computation and cryptography. Hence, every classical cryptographic protocol based on Boolean function might have a counterpart in quantum information theory via hypergraph states. The interested reader may approach further in this direction.

	\section*{Acknowledgment}
	
		The author is thankful to Prof. Prasanta K. Panigrahi for a number of discussions during this work.


\begin{thebibliography}{10}
	
	\bibitem{bennett1992communication}
	Charles~H Bennett and Stephen~J Wiesner.
	\newblock Communication via one-and two-particle operators on
	einstein-podolsky-rosen states.
	\newblock {\em Physical review letters}, 69(20):2881, 1992.
	
	\bibitem{nielsen2002quantum}
	Michael~A Nielsen and Isaac Chuang.
	\newblock Quantum computation and quantum information, 2002.
	
	\bibitem{bennett1993teleporting}
	Charles~H Bennett, Gilles Brassard, Claude Cr{\'e}peau, Richard Jozsa, Asher
	Peres, and William~K Wootters.
	\newblock Teleporting an unknown quantum state via dual classical and
	einstein-podolsky-rosen channels.
	\newblock {\em Physical review letters}, 70(13):1895, 1993.
	
	\bibitem{devitt2013quantum}
	Simon~J Devitt, William~J Munro, and Kae Nemoto.
	\newblock \href{https://arxiv.org/abs/0905.2794}{Quantum error correction for
		beginners}.
	\newblock {\em Reports on Progress in Physics}, 76(7):076001, 2013.
	
	\bibitem{mcmahon2007quantum}
	David McMahon.
	\newblock {\em Quantum computing explained}.
	\newblock John Wiley \& Sons, 2007.
	
	\bibitem{wilde2013quantum}
	Mark~M Wilde.
	\newblock {\em \href{https://arxiv.org/abs/1106.1445}{Quantum information
			theory}}.
	\newblock Cambridge University Press, 2013.
	
	\bibitem{benenti2007principles}
	Giuliano Benenti, Giulio Casati, and Giuliano Strini.
	\newblock {\em Principles of Quantum Computation and Information}.
	\newblock World Scientific Publishing Company, 2007.
	
	\bibitem{horodecki2009quantum}
	Ryszard Horodecki, Pawe{\l} Horodecki, Micha{\l} Horodecki, and Karol
	Horodecki.
	\newblock \href{https://arxiv.org/abs/quant-ph/0702225}{Quantum entanglement}.
	\newblock {\em Reviews of modern physics}, 81(2):865, 2009.
	
	\bibitem{qu2013encoding}
	Ri~Qu, Juan Wang, Zong-shang Li, and Yan-ru Bao.
	\newblock \href{https://arxiv.org/abs/1211.3911}{Encoding hypergraphs into
		quantum states}.
	\newblock {\em Physical Review A}, 87(2):022311, 2013.
	
	\bibitem{rossi2013quantum}
	Matteo Rossi, Marcus Huber, Dagmar Bru{\ss}, and Chiara Macchiavello.
	\newblock \href{https://arxiv.org/abs/1211.5554}{Quantum hypergraph states}.
	\newblock {\em New Journal of Physics}, 15(11):113022, 2013.
	
	\bibitem{hein2004multiparty}
	Marc Hein, Jens Eisert, and Hans~J Briegel.
	\newblock \href{https://arxiv.org/abs/quant-ph/0307130}{Multiparty entanglement
		in graph states}.
	\newblock {\em Physical Review A}, 69(6):062311, 2004.
	
	\bibitem{anders2006fast}
	Simon Anders and Hans~J Briegel.
	\newblock \href{https://arxiv.org/abs/quant-ph/0504117}{Fast simulation of
		stabilizer circuits using a graph-state representation}.
	\newblock {\em Physical Review A}, 73(2):022334, 2006.
	
	\bibitem{van2005local}
	Maarten Van~den Nest, Jeroen Dehaene, and Bart De~Moor.
	\newblock \href{https://arxiv.org/abs/quant-ph/0411115}{Local unitary versus
		local Clifford equivalence of stabilizer states}.
	\newblock {\em Physical Review A}, 71(6):062323, 2005.
	
	\bibitem{briegel2001persistent}
	Hans~J Briegel and Robert Raussendorf.
	\newblock \href{https://arxiv.org/abs/quant-ph/0004051}{Persistent entanglement
		in arrays of interacting particles}.
	\newblock {\em Physical Review Letters}, 86(5):910, 2001.
	
	\bibitem{nielsen2006cluster}
	Michael~A Nielsen.
	\newblock \href{https://arxiv.org/abs/quant-ph/0504097v2}{Cluster-state quantum
		computation}.
	\newblock {\em Reports on Mathematical Physics}, 57(1):147--161, 2006.
	
	\bibitem{qu2013multipartite}
	Ri~Qu, Zong-shang Li, Juan Wang, and Yan-ru Bao.
	\newblock \href{https://arxiv.org/abs/1301.3576}{Multipartite entanglement and
		hypergraph states of three qubits}.
	\newblock {\em Physical Review A}, 87(3):032329, 2013.
	
	\bibitem{wagner2018analysis}
	Thomas Wagner, Hermann Kampermann, and Dagmar Bru{\ss}.
	\newblock \href{https://arxiv.org/abs/1711.00295}{Analysis of quantum error
		correction with symmetric hypergraph states}.
	\newblock {\em Journal of Physics A: Mathematical and Theoretical},
	51(12):125302, 2018.
	
	\bibitem{dutta2019permutation}
	Supriyo Dutta, Ramita Sarkar, and Prasanta~K Panigrahi.
	\newblock Permutation symmetric hypergraph states and multipartite quantum
	entanglement.
	\newblock {\em International Journal of Theoretical Physics},
	58(11):3927--3944, 2019.
	
	\bibitem{zhou2022entanglement}
	You Zhou and Alioscia Hamma.
	\newblock \href{https://arxiv.org/abs/2110.07158}{Entanglement of random
		hypergraph states}.
	\newblock {\em Physical Review A}, 106(1):012410, 2022.
	
	\bibitem{amouzou2022entanglement}
	Gr{\^a}ce Amouzou, Jeoffrey Boffelli, Hamza Jaffali, Kossi Atchonouglo, and
	Fr{\'e}d{\'e}ric Holweck.
	\newblock \href{https://arxiv.org/abs/2010.03217}{Entanglement and nonlocality
		of four-qubit connected hypergraph states}.
	\newblock {\em International Journal of Quantum Information}, 20(03):2250001,
	2022.
	
	\bibitem{sarkar2022geometry}
	Ramita Sarkar, Shreya Banerjee, Subhasish Bag, and Prasanta~K Panigrahi.
	\newblock Geometry of distributive multiparty entanglement in 4- qubit
	hypergraph states.
	\newblock {\em IET Quantum Communication}, 3(1):72--84, 2022.
	
	\bibitem{moore2019quantum}
	Darren~W Moore.
	\newblock \href{https://arxiv.org/abs/1909.03871}{Quantum hypergraph states in
		continuous variables}.
	\newblock {\em Physical Review A}, 100(6):062301, 2019.
	
	\bibitem{ghio2017multipartite}
	Maddalena Ghio, Daniele Malpetti, Matteo Rossi, Dagmar Bru{\ss}, and Chiara
	Macchiavello.
	\newblock \href{https://arxiv.org/abs/1703.00429v2}{Multipartite entanglement
		detection for hypergraph states}.
	\newblock {\em Journal of Physics A: Mathematical and Theoretical},
	51(4):045302, 2017.
	
	\bibitem{qu2013relationship}
	Ri~Qu, Yi-ping Ma, Bo~Wang, and Yan-ru Bao.
	\newblock \href{https://arxiv.org/abs/1304.6275}{Relationship among locally
		maximally entangleable states, W states, and hypergraph states under local
		unitary transformations}.
	\newblock {\em Physical Review A}, 87(5):052331, 2013.
	
	\bibitem{qu2014entropic}
	Ri~Qu, Yi-ping Ma, Yan-ru Bao, Juan Wang, and Zong-shang Li.
	\newblock \href{https://arxiv.org/abs/1305.0662}{Entropic measure and
		hypergraph states}.
	\newblock {\em Quantum information processing}, 13(2):249--258, 2014.
	
	\bibitem{guhne2014entanglement}
	Otfried G{\"u}hne, Marti Cuquet, Frank~ES Steinhoff, Tobias Moroder, Matteo
	Rossi, Dagmar Bru{\ss}, Barbara Kraus, and Chiara Macchiavello.
	\newblock \href{https://arxiv.org/abs/1404.6492}{Entanglement and nonclassical
		properties of hypergraph states}.
	\newblock {\em Journal of Physics A: Mathematical and Theoretical},
	47(33):335303, 2014.
	
	\bibitem{zhu2019efficient}
	Huangjun Zhu and Masahito Hayashi.
	\newblock \href{arxiv.org/abs/1806.05565}{Efficient verification of hypergraph
		states}.
	\newblock {\em Physical Review Applied}, 12(5):054047, 2019.
	
	\bibitem{xiong2018qudit}
	Fei-Lei Xiong, Yi-Zheng Zhen, Wen-Fei Cao, Kai Chen, and Zeng-Bing Chen.
	\newblock \href{https://arxiv.org/abs/1701.07733}{Qudit hypergraph states and
		their properties}.
	\newblock {\em Physical Review A}, 97(1):012323, 2018.
	
	\bibitem{toth2005detecting}
	G{\'e}za T{\'o}th and Otfried G{\"u}hne.
	\newblock \href{https://arxiv.org/abs/quant-ph/0405165}{Detecting genuine
		multipartite entanglement with two local measurements}.
	\newblock {\em Physical review letters}, 94(6):060501, 2005.
	
	\bibitem{takeuchi2019quantum}
	Yuki Takeuchi, Tomoyuki Morimae, and Masahito Hayashi.
	\newblock \href{https://arxiv.org/abs/1809.07552}{Quantum computational
		universality of hypergraph states with Pauli-X and Z basis measurements}.
	\newblock {\em Scientific reports}, 9(1):1--14, 2019.
	
	\bibitem{takeuchi2018verification}
	Yuki Takeuchi and Tomoyuki Morimae.
	\newblock \href{https://arxiv.org/abs/1709.07575}{Verification of many-qubit
		states}.
	\newblock {\em Physical Review X}, 8(2):021060, 2018.
	
	\bibitem{gachechiladze2019changing}
	Mariami Gachechiladze, Otfried G{\"u}hne, and Akimasa Miyake.
	\newblock \href{https://arxiv.org/abs/1805.12093}{Changing the circuit-depth
		complexity of measurement-based quantum computation with hypergraph states}.
	\newblock {\em Physical Review A}, 99(5):052304, 2019.
	
	\bibitem{tao2022verification}
	Hong Tao, Xiaoqian Zhang, Lei Shao, and Xiaoqing Tan.
	\newblock \href{https://arxiv.org/abs/2203.09989}{Verification of colorable
		hypergraph states with stabilizer test}.
	\newblock {\em arXiv preprint arXiv:2203.09989}, 2022.
	
	\bibitem{banerjee2020quantum}
	Shreya Banerjee, Arghya Mukherjee, and Prasanta~K Panigrahi.
	\newblock Quantum blockchain using weighted hypergraph states.
	\newblock {\em Physical Review Research}, 2(1):013322, 2020.
	
	\bibitem{galindo2002information}
	Alberto Galindo and Miguel~Angelo Martin-Delgado.
	\newblock \href{https://arxiv.org/abs/quant-ph/0112105}{Information and
		computation: Classical and quantum aspects}.
	\newblock {\em Reviews of Modern Physics}, 74(2):347, 2002.
	
	\bibitem{crawley1973algebraic}
	Peter Crawley and Robert~Palmer Dilworth.
	\newblock {\em Algebraic theory of lattices}.
	\newblock Prentice Hall, 1973.
	
	\bibitem{lidl2012applied}
	Rudolf Lidl and G{\"u}nter Pilz.
	\newblock {\em Applied abstract algebra}.
	\newblock Springer Science \& Business Media, 2012.
	
	\bibitem{bretto2013hypergraph}
	Alain Bretto.
	\newblock Hypergraph theory.
	\newblock {\em An introduction. Mathematical Engineering. Cham: Springer},
	2013.
	
	\bibitem{crama2011boolean}
	Yves Crama and Peter~L Hammer.
	\newblock {\em Boolean functions: Theory, algorithms, and applications}.
	\newblock Cambridge University Press, 2011.
	
	\bibitem{brown2005searching}
	Iain~DK Brown, Susan Stepney, Anthony Sudbery, and Samuel~L Braunstein.
	\newblock Searching for highly entangled multi-qubit states.
	\newblock {\em Journal of Physics A: Mathematical and General}, 38(5):1119,
	2005.
	
\end{thebibliography}

\end{document}